\documentclass{amsart}

\usepackage{amssymb,amsmath,amsfonts, amsthm}
\usepackage{amscd,enumerate,graphicx}
\newtheorem{theorem}{Theorem}[section]
\newtheorem{proposition}[theorem]{Proposition}
\newtheorem{lemma}[theorem]{Lemma}

\newtheorem{corollary}[theorem]{Corollary}
\theoremstyle{remark}

\theoremstyle{definition}

\newcommand{\C}{\Bbb C}

\newcommand{\tr}{{\mathrm{tr}\,}}
\newcommand{\la}{\langle}
\newcommand{\ra}{\rangle}

\newcommand{\fb}{\mathfrak{b}}

\numberwithin{equation}{section}

\title{On minimality of two-bridge knots}

\author{Fumikazu Nagasato}
\address{Department of Mathematics, 
Meijo University, 
Tempaku, Nagoya, 468-8502, Japan}
\email{fukky@meijo-u.ac.jp}

\author{Masaaki Suzuki}
\address{Department of Frontier Media Science, 
Meiji University, 
4-21-1 Nakano, Nakano-ku, Tokyo, 164-8525, Japan}
\email{macky@fms.meiji.ac.jp}

\author{Anh T. Tran}
\address{Department of Mathematical Sciences, 
The University of Texas at Dallas, Richardson,
TX 75080, USA}
\email{att140830@utdallas.edu}

\begin{document}

\maketitle

\begin{abstract}
A knot is called minimal if its knot group admits epimorphisms 
 onto the knot groups of only the trivial knot and itself. 
In this paper,  we determine which two-bridge knot $\fb(p,q)$ is minimal 
where $q \leq 6$ or $p \leq 100$.
\end{abstract}

\section{Introduction}

Recently, many papers have investigated epimorphisms between knot groups. 
In particular, Simon's conjecture, which was one of our main interests, 
was settled affirmatively in \cite{agolliu}. 
Namely, every knot group maps onto at most finitely many knot groups. 
Then the next problem will be to determine 
the number of knot groups onto which a given knot group maps. 
Silver and Whitten in \cite{silverwhitten} gave a necessary and sufficient condition 
for admitting an epimorphism from the knot group of a torus knot. 
Ohtsuki, Riley, and Sakuma in \cite{ORS} studied 
a systematic construction of epimorphisms between two-bridge knot groups. 
In \cite{kitano-suzuki1} and \cite{HKMS} all the pairs of prime knots with up to $11$ crossings 
which admit meridional epimorphisms between their knot groups were determined. 
However, it is not easy to determine 
whether there exists an epimorphism between knot groups in general. 

A knot is called {\it minimal} if its knot group admits epimorphisms 
 onto the knot groups of only the trivial knot and itself. 
It is already known that infinitely many knots are minimal. 
For example, it is proved in \cite{kitano-suzuki3} that 
all the prime knots with up to $6$ crossings are minimal. 
This result has been extended to two-bridge knots with up to $8$ crossings in \cite{szk}.
Besides, the previous papers \cite{Burde}, \cite{MPV}, \cite{N}, and \cite{NT} 
showed that twist knots and $\fb(p,3)$ are minimal. 

In this paper, we will discuss the minimality of two-bridge knots. 
To be precise, we determine which two-bridge knot $\fb(p,q)$ is minimal 
where $q \leq 6$ or $p \leq 100$, 
in Theorem \ref{thm:mainthmA} and Theorem \ref{thm:mainthmB}. 
For this purpose, 
we show two sufficient conditions on the minimality of a knot. 
One is given by the  $SL_2(\C)$-character variety of a knot 
and the other heavily relies on the irreducibility of the Alexander polynomial. 

%\begin{theorem}
%The two-bridge knot $\fb(p,q)$, where $q \leq 6$, is not minimal 
%if and only if one of the following conditions is satisfied: \\
%$(1)$ $p$ is not prime and $q = 1$, \\
%$(2)$ $p = 30 k \pm 3$ and $q = 5$.
%\end{theorem}
%
%\begin{theorem}
%The two-bridge knot $\fb(p,q)$, where $p \leq 100$, is not minimal 
%if and only if $(p,q)$ is one of the following pairs: 
%\begin{align*}
%& (9,1), (15,1), (21,1), (27,1), (27,5), (33,1), (33,5), (35,1), (39,1), (39,7), (45,1), \\
%& (45,7), (45,19), (49,1), (51,1), (55,1), (57,1), (57,5), (63,1), (63,5), (63,11), (65,1), \\
%& (69,1), (69,11), (69,19), (75,1), (75,13), (75,29), (77,1), (81,1), (81,7), (81,13), \\
%& (85,1), (85,9), (85,38), (87,1), (87,5), (87,7), (91,1), (93,1), (93,5), (93,11), (95,1), \\
%& (95,9), (99,1), (99,17), (99,29).
%\end{align*}
%\end{theorem}

\section{Some sufficient conditions}\label{sect:conditions}

In this section, 
we show some sufficient conditions on the minimality of a knot. 
%They are very useful to show that a knot is minimal. 
First, we consider the nonabelian  $SL_2(\C)$-character variety of a knot, 
which is the non-abelian part of the $SL_2(\C)$-character variety. 

For a knot $K$ in $S^3$, we denote by $G(K)$ its knot group and $X(K)$ its $SL_2(\C)$-character variety. 
For brevity, in the rest of the paper, character varieties always mean $SL_2(\C)$-character varieties. 

\begin{proposition}\label{prop:secondcondition}
If the nonabelian character variety of a hyperbolic knot $K$ is irreducible, 
then $K$ is minimal. 
\end{proposition}

The previous paper 
\cite[Theorem 4.4]{BBRW} (see also \cite[Corollary 7.1]{ORS}) shows Proposition 
\ref{prop:secondcondition}. However, its proof would be a little indirect for it. 
In \cite[Appendix B]{N}, we showed a straightforward proof of Proposition 
\ref{prop:secondcondition}. Here we review this proof. 

\begin{proof}[Proof of Proposition \ref{prop:secondcondition}]
Suppose there exists a non-trivial knot $K'$ in $S^3$ such that $K'\neq K$ and 
there exists a surjective group homomorphism $\phi: G(K) \to G(K')$. 
Then $\phi$ induces the injection $\phi^*:X(K') \to X(K)$ between 
the character varieties $X(K')$ and $X(K)$ given by 
$\phi^*(\chi_{\rho'}):=\chi_{\rho \circ \phi}$ (see \cite[Lemma 2.1]{BBRW}). 

Now $X(K)$ is a union of the irreducible components $X^{\rm red}(K)$ 
consisting entirely of the reducible characters and 
the irreducible components $X^{\rm irr}(K)$ including the irreducible characters. 
Since $K$ is hyperbolic, $X^{\rm irr}(K)$ includes the irreducible component $X_0$ 
containing a discrete faithful character $\chi_{\rho_0}$ 
coming from the holonomy representation. Note that $\dim_{\C}(X_0)=1$. 
By the assumption of Proposition \ref{prop:secondcondition}, $X^{\rm irr}(K)$ has only 
a single irreducible component and thus $X^{\rm irr}(K)=X_0$. 
It is obvious that $\phi^*$ maps irreducible (resp. reducible) characters in $X(K')$ 
to irreducible (resp. reducible) characters in $X(K)$. 
This indicates that $X(K')$ must be a union of the irreducible components 
$X^{\rm red}(K')$ consisting entirely of the reducible characters 
and a single (complex) 1-dimensional irreducible component $X'$ 
containing the irreducible characters because $\phi^*$ is an injection. 
(Note that the character variety of a non-trivial knot always has an irreducible character.)
Moreover, $\phi^*(X')=X_0$ as $\phi^*$ is a closed map in Zariski topology, 
which can be shown as follows (see also \cite[Lemma 2.1]{BB}). 
Suppose that the set $C:=X_0-\phi^*(X')$ is not empty. 
Then there exists a point $x_0$ in $C$ and a sequence $\{\chi_{\rho'_n}\}$ in $X'$, 
called {\it blow-up}, such that for some $h \in G(K')$
\[ 
\left|\lim_{n \to \infty}\chi_{\rho'_n}(h)\right|=\infty,\hspace*{1cm}
\lim_{n \to \infty} \phi^*(\chi_{\rho'_n})=x_0.
\]
For example, we can consider a sequence $\{\chi_{\rho_n}\}$ 
in $\phi^*(X')=X_0-C$ with 
\[\lim_{n \to \infty}\chi_{\rho_n}=x_0\] 
because $\dim_{\C}(X_0)=1$. Then taking the preimage of $\chi_{\rho_n}$ 
under $\phi^*$, we can obtain a blow-up $\{\chi_{\rho'_n}\}$.

By definition, the above behavior of $\phi^*(\chi_{\rho'_n})$ at infinity 
can be transformed as follows: 
\[
\left|\lim_{n \to \infty} \chi_{\rho'_n}(\phi(g))\right|=|x_0(g)|, 
\]
for any $g \in G(K)$. Note that since $x_0$ is in $X_0$, 
the right side of the above equation does not diverge for any $g \in G(K)$. 
As $\phi:G(K) \to G(K')$ is surjective, for some $g_0 \in G(K)$ 
we have $\phi(g_0)=h$ and thus the equality 
\[
\left|\lim_{n \to \infty} \chi_{\rho'_n}(h)\right|=|x_0(g_0)|. 
\]
holds. The left side of the equality diverges, 
meanwhile the right side converges, a contradiction. 

Now we have $\phi^*(X')=X_0$. 
In this situation, there must exist an irreducible character $\chi_{\rho'} \in X'$ 
such that $\phi^*(\chi_{\rho'})=\chi_{\rho_0}$, i.e., 
$\chi_{\rho' \circ \phi}=\chi_{\rho_0}$. 
Since $\rho' \circ \phi$ and $\rho_0$ are irreducible, they are conjugate. 
So there exists a matrix $A \in SL_2(\C)$ such that 
$A^{-1} \rho'(\phi(g)) A=\rho_0(g)$ for any $g \in G(K)$. 
Here $\rho_0$ is faithful. Hence $\phi:G(K) \to G(K')$ must be bijective 
and thus a group isomorphism. 

In this situation, $K'$ cannot be a prime knot, 
because if $K'$ is prime then $G(K) \cong G(K')$ means $K=K'$, a contradiction. 
This completes the proof of Proposition \ref{prop:secondcondition} for prime knots.
In addition, $K'$ cannot be a composite knot by \cite[Lemma 2]{FW}. 
\end{proof}

We have another sufficient condition for a knot to be minimal, 
which heavily relies on the Alexander polynomial. 
It is well-known that 
if there exists an epimorphism from the knot group of $K$ onto that of $K'$, 
then the Alexander polynomial $\Delta_K(t)$ of $K$ is divisible by that of $K'$. 
This property gives us a useful criterion to show the non-existence of an epimorphism 
between knot groups. 
Moreover, we can refine this condition on two-bridge knots. 

\begin{proposition}\label{prop:firstcondition}
Let $K$ be a two-bridge knot. 
If the Alexander polynomial of $K$ is irreducible over the integers ${\mathbb Z}$, 
then $K$ is minimal. 
\end{proposition}

A two-bridge knot can be represented by a fraction $q/p$. 
Then we denote by $\fb(p,q)$ the two-bridge knot, where $p$ and $q$ are coprime 
(see \cite{bzh}, \cite{murasugi} for details). 
Using this presentation, 
%Remark that 
Schubert classified two-bridge knots as follows. 

\begin{theorem}[Schubert] 
Let $\fb(p,q)$ and $\fb(p',q')$ be two-bridge knots. 
These knots are equivalent if and only if 
the following conditions hold. 
\begin{enumerate}
\item[(1)] $p = p'$. 
\item[(2)] Either $q \equiv \pm q' \pmod{p}$ or $q q' \equiv \pm 1 \pmod{p}$. 
\end{enumerate}
\end{theorem}

In order to prove Proposition \ref{prop:firstcondition}, 
we recall the following remarkable result. 

\begin{theorem}[Boileau-Boyer \cite{BB}]\label{thm:bb}
%Let $K$ be the two-bridge knot $\fb(p,q)$. 
If there exists an epimorphism from $G(\fb(p,q))$ onto another knot group $G(K)$, 
then $K$ is the trivial knot or also a two-bridge knot $\fb(p',q')$, 
where $p = k p'$ and $k > 1$. 
\end{theorem}

\begin{proof}[Proof of Proposition \ref{prop:firstcondition}]
Suppose that the knot group of a two-bridge knot $\fb(p,q)$ 
admits an epimorphism onto the knot group of another non-trivial knot $K$.  
By Theorem \ref{thm:bb}, 
$K$ is a two-bridge knot $\fb(p',q')$, 
where, in particular, $p'$ is less than $p$. 
It is easy to see that 
if $G(\fb(p,q))$ admits an epimorphism onto $G(K)$, 
then the Alexander polynomial of $\fb(p,q)$ is divisible by that of $K$. 
However, by assumption, the Alexander polynomial of $\fb(p,q)$ is irreducible over ${\mathbb Z}$. 
Therefore the Alexander polynomial of $K$ is the same as that of $\fb(p,q)$. 
Note that the Alexander polynomial of a two-bridge knot is not trivial. 
On the other hand, it is known that the determinant $|\Delta_{\fb(p,q)}(-1)|$ 
of a two-bridge knot $\fb(p,q)$ is $p$. 
%where we denote by $\Delta_K (t)$ the Alexander polynomial of $K$. 
This implies the Alexander polynomials of $\fb(p,q)$ and $K$ are not the same. 
This is a contradiction. 
\end{proof}

\begin{corollary}\label{cor:alexpolyofdeg2}
If the degree of the Alexander polynomial of a two-bridge knot $K$ is $2$, 
then $K$ is minimal. 
\end{corollary}

\section{Main theorem I}

In this section, we determine which two-bridge knot $\fb(p,q)$ is minimal where $q \leq 6$. 

First of all, we can determine whether two-bridge knots $\fb(p,1)$ are minimal or not, 
since $\fb(p,1)$ is a torus knot and the existence of an epimorphism from a torus knot group 
is studied in \cite{silverwhitten}. 
Namely, $\fb(p,1)$ is minimal if and only if $p$ is prime. 

Next, two-bridge knots $\fb(2k + 1,2)$ are minimal, 
since these knots are twist knots and their minimalities were already shown 
in \cite{Burde}, \cite{MPV}, \cite{N}, and \cite{NT}. 
Furthermore, two-bridge knots 
$\fb(3k+ 1,3)$, $\fb(3k+ 2,3)$, $\fb(4k+ 1,4)$, and $\fb(4k + 3,4)$  
are also minimal as follows. 
These knots are always double twist knots and we have the following proposition. 

\begin{figure}[htbp]
$$
\begin{minipage}{8cm}\includegraphics[width=\hsize]{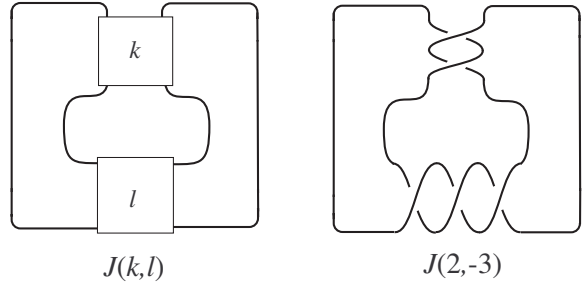}\end{minipage}
$$
\caption{Double twist knots $J(k,l)$ and $J(2,-3)$.}
\label{fig_jkl}
\end{figure}

\begin{proposition}
Double twist knots $J(k,l)$, where $k l$ is even, are minimal. 
\end{proposition}

\begin{proof}
It was shown in \cite{MPV} that if $k \neq l$, 
the nonabelian character variety of the double twist knot $J(k,l)$ is irreducible, 
in which case $J(k,l)$ is minimal by Proposition \ref{prop:secondcondition}. 
Now we consider the double twist knot $J(2k,2k)$. 
By \cite{MPV}, the nonabelian character variety of $J(2k,2k)$ is reducible. 
However, 
it is easy to see that the Alexander polynomial of $J(2k,2k)$ is 
\[
\Delta_{J(2k,2k)} (t) = k^2 t^2 - (2 k^2 - 1) + k^2,
\]
which is irreducible over ${\mathbb Z}$. 
By Proposition \ref{prop:firstcondition}, we obtain the statement.
\end{proof}

In general, 
two-bridge knots $\fb(kq+1,q)$ and $\fb(kq-1,q)$ are 
the double twist knots as depicted in Figure \ref{fig_bkp+1}. 

\begin{figure}[htbp]
$$
\begin{minipage}{6cm}\includegraphics[width=\hsize]{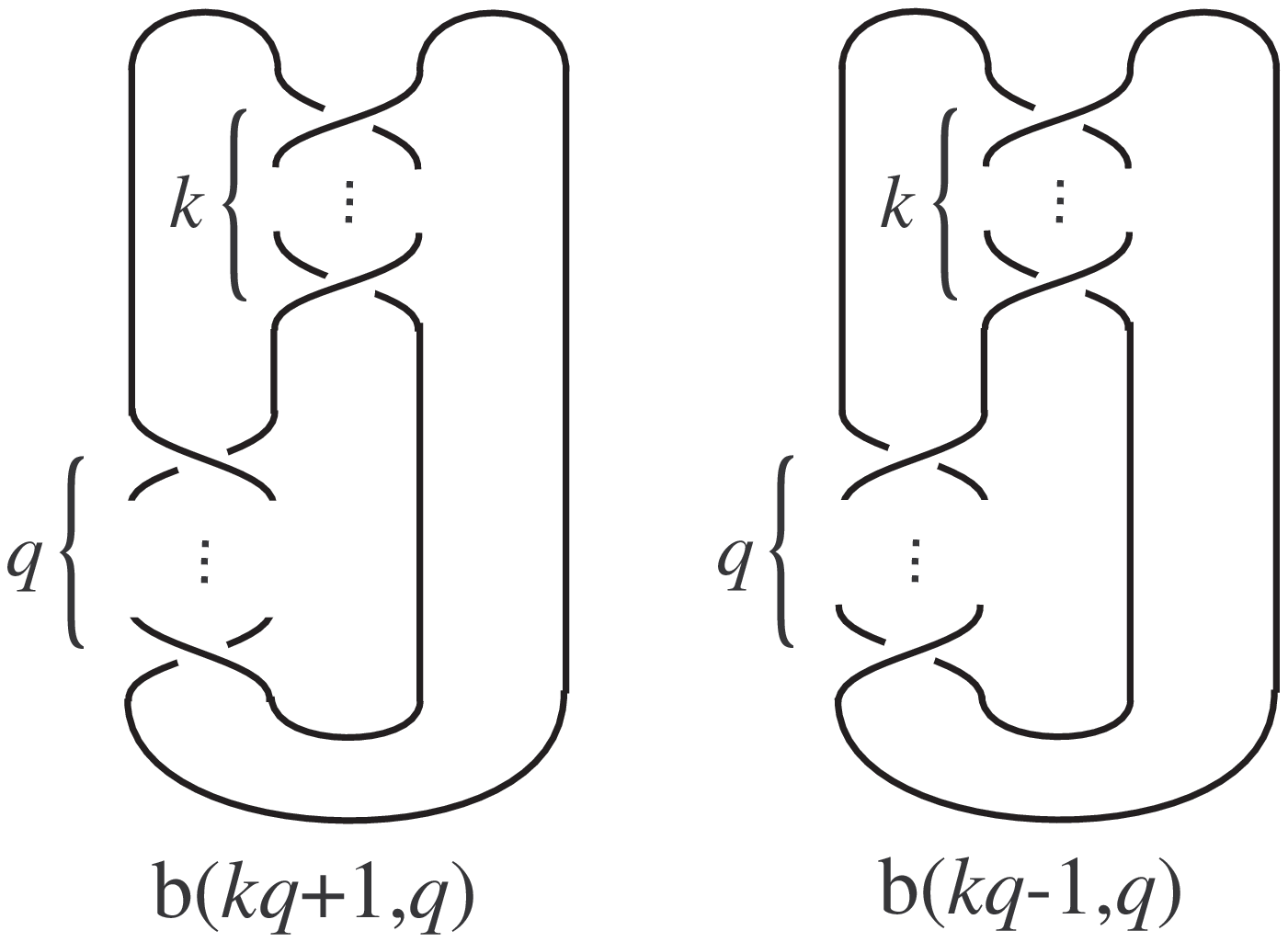}\end{minipage}
$$
\caption{2-bridge knots $\fb(kq+1,q)$ and $\fb(kq-1,q)$.}
\label{fig_bkp+1}
\end{figure}

\begin{corollary}\label{cor:kpplusminusone}
Two-bridge knots $\fb(kq+1,q)$ and $\fb(kq-1,q)$ are minimal. 
\end{corollary}

Next, we discuss two-bridge knots $\fb(p,5)$. 
As stated above, two-bridge knots $\fb(5k + 1,5)$ and $\fb(5k + 4,5)$ are double twist knots 
and hence they are minimal. 
For two-bridge knots $\fb(5k+2,5)$ and $\fb(5k+3,5)$, 
we need to investigate their nonabelian character varieties. 

\begin{theorem} \label{thm} 
For any integer $n \ge 0$, the following holds:  
\begin{enumerate}
\item[(1)] The nonabelian character variety of the two-bridge knot $\fb(5(2n+1) + 2,5)$ 
is reducible if and only if $n \equiv 2 \pmod{3}$.
\item[(2)] The nonabelian character variety of the two-bridge knot $\fb(5(2n+1) - 2,5)$ 
is reducible if and only if $n \equiv 0 \pmod{3}$.
\end{enumerate}
\end{theorem}

\begin{corollary}
Two-bridge knots $\fb(5k + 3,5)$ and $\fb(5k - 3,5)$ are minimal, if $k \not\equiv 0 \pmod{6}$.
\end{corollary}

The proof of (2) of Theorem \ref{thm} is similar to that of (1) and 
so we present only the proof of (1). 

Let $K$ denote the two-bridge knot $\fb(5(2n+1)+2,5)$. 
The knot group of $K$ admits the following standard two-generator presentation: 
%of a two-bridge knot group
$$\pi_1(K) = \la a, b \mid wa =  b w \ra$$
where $$w= a (ba)^nb^{-1}(ba)^{-n}  a (ba)^n b (ba)^{-n} a^{-1} (ba)^n b.$$

%\subsection{Chebyshev polynomials}
We make use of the Chebyshev polynomials 
to describe the nonabelian character variety. 
Let $S_k(y)$ be the Chebychev polynomials of the second kind defined by 
$S_0(y)=1$, $S_1(y)=y$ and $S_{k}(y)=yS_{k-1}(y)-S_{k-2}(y)$ for all integers $k$. 

Note that $S_k(2)=k+1$ and $S_k(-2)=(-1)^k (k+1)$. 
Moreover if $y=u+u^{-1}$, where $u \not= \pm 1$, then 
$S_k(y)=\frac{u^{k+1}-u^{-(k+1)}}{u-u^{-1}}$.

\begin{lemma}[cf. \cite{N}] \label{chev}
For all $k \ge 1$ we have 
$$S_k(y) = \prod_{j=1}^{k} \Big( y-2\cos \frac{j\pi}{k+1} \Big).$$
\end{lemma}

\begin{lemma} [cf. \cite{Tran}] \label{gcd}
For all integers $k$ we have
$$S^2_k(y) + S^2_{k-1}(y) - y S_k(y) S_{k-1}(y) =1.$$
Consequently, we have $\gcd(S_k(y), S_{k-1}(y)) = 1$ in $\C[y]$.
\end{lemma}

\begin{lemma} [cf. \cite{Tran}] \label{power}
Suppose $M = \left[ \begin{array}{cc}
a & b \\
c & d \end{array} \right] \in SL_2(\C)$. Then, for all integers $k$, we have 
\begin{eqnarray*}
M^k &=& \left[ \begin{array}{cc}
S_{k}(y) - d S_{k-1}(y) & b S_{k-1}(y) \\
c S_{k-1}(y) & S_{k}(y) - a S_{k-1}(y) \end{array} \right]
\end{eqnarray*}
where $y = \tr M$.
\end{lemma}

%\subsection{Nonabelian character variety}
Suppose $\rho: \pi_1(K) \to SL_2(\C)$ is a nonabelian representation. 
By Riley's result \cite{riley}, 
up to conjugation we may assume that $$\rho(a) = \begin{pmatrix}
s & 1\\
0 & s^{-1}
\end{pmatrix} \quad \text{and} \quad \rho(b) = \begin{pmatrix}
s & 0\\
z - s^2 - s^{-2} & s^{-1}
\end{pmatrix},$$ 
where $(s,z) \in \C^2$ satisfies the matrix equation $\rho(wa) - \rho(bw) = O$. 
Note that $z = \tr \rho(ba)$. 

We first compute $\rho((ba)^n)$. 

\begin{lemma} \label{(ba)^m}
We have
$$
\rho((ba)^n) = \begin{pmatrix}
S_n(z) - s^{-2} S_{n-1}(z) & s^{-1} S_{n-1}(z)\\
s^{-1}(z - s^2 - s^{-2}) S_{n-1}(z) & S_n(z) - (z - s^{-2}) S_{n-1}(z)
\end{pmatrix}.
$$
\end{lemma}

\begin{proof}
The lemma follows from 
$$
\rho(ba) = \begin{pmatrix}
z - s^{-2} & s^{-1}\\
s^{-1}(z - s^2 - s^{-2})  & s^{-2}
\end{pmatrix}
$$
and Lemma \ref{power}.
\end{proof}

We now compute $\rho(wa) - \rho(bw)$. Let $x= \tr \rho(a) = \tr \rho(b) = s + s^{-1}$, $X = S_n(z)$ and $Y = S_{n-1}(z)$. 

\begin{proposition} \label{char}
We have
$$
\rho(wa) - \rho(bw) = \begin{pmatrix}
0 & R(x,z)\\
-(z + 2 - x^2) R(x,z) & 0
\end{pmatrix}
$$
where $R(x,z) =\alpha x^4 + \beta x^2 + \gamma$ and $\alpha, \beta, \gamma \in \C[z]$ given by
\begin{eqnarray*}
\alpha &=& (z-2) X^2  (X-Y)^3, \\
\beta &=& - (z-2) X  (X-Y) (2 X^3 z+X^3-X^2 Y z-6 X^2 Y+X Y^2+2 Y^3), \\
\gamma &=& (X-Y)^5 + (z-2) (X^5 z^2+X^5 z-5 X^4 Y z-5 X^4 Y+10 X^3 Y^2-Y^5).
\end{eqnarray*}
\end{proposition}

\begin{proof}
By Lemma \ref{(ba)^m} we have
$$
\rho((ba)^n) = \begin{pmatrix}
X - s^{-2} Y & s^{-1} Y\\
s^{-1}(z - s^2 - s^{-2}) Y & X - (z - s^{-2}) Y
\end{pmatrix}.
$$
Since $w= a (ba)^nb^{-1}(ba)^{-n}  a (ba)^n b (ba)^{-n} a^{-1} (ba)^n b$, 
by a direct computation (using computer) we have
$$\rho(wa) - \rho(bw) = \begin{pmatrix}
0 & R(x,z)\\
-(z  - s^2 - s^{-2}) R(x,z) & 0
\end{pmatrix}
$$
where
$$R(x,z) = \alpha' (s^4 + s^{-4}) + \beta' (s^2 + s^{-2}) + \gamma'$$
and
\begin{eqnarray*}
\alpha' &=& X^5 z-2 X^5-3 X^4 Y z+6 X^4 Y+3 X^3 Y^2 z-6 X^3 Y^2-X^2 Y^3 z+2 X^2 Y^3, \\
\beta' &=& -2 X^5 z^2+7 X^5 z-6 X^5+3 X^4 Y z^2-11 X^4 Y z +10 X^4 Y - X^3 Y^2 z^2 \\
       && + \, 7 X^3 Y^2 z-10 X^3 Y^2-5 X^2 Y^3 z +10 X^2 Y^3 + 2 X Y^4 z-4 X Y^4, \\
\gamma' &=& X^5 z^3-5 X^5 z^2+10 X^5 z-7 X^5+X^4 Y z^2-11 X^4 Y z+13 X^4 Y \\
              && - \, 2 X^3 Y^2 z^2+18 X^3 Y^2 z-18 X^3 Y^2-8 X^2 Y^3 z+6 X^2 Y^3\\
              && + \, 4 X Y^4 z-3 X Y^4-Y^5 z+Y^5.
\end{eqnarray*}

By substituting $s^4 + s^{-4} = x^4 - 4 x^2 +2$ and $s^2 + s^{-2} = x^2 - 2$, 
we obtain the desired formula for $R(x,z)$.
\end{proof}

Proposition \ref{char} implies that the matrix equation $\rho(wa) - \rho(bw) = O$ 
is equivalent to a single polynomial equation $R(x,z)=0$. 
This equation determines the characters of nonabelian $SL_2(\C)$-representations, 
i.e. the nonabelian $SL_2(\C)$-character variety, of $K = \fb(5(2n+1)+2, 5)$.

Now let us consider the factorization of $R(x,z) \in \C[x,z]$. 
We first focus on the coefficients $\alpha, \beta, \gamma$ of $R(x,z)$. 

\begin{lemma} \label{lem1}
If $n \not\equiv 2 \pmod{3}$ then $\gcd(\alpha, \beta, \gamma) = 1$ in $\C[z]$.
If $n \equiv 2 \pmod{3}$ then $\gcd(\alpha, \beta, \gamma) = z-1$ in $\C[z]$.
\end{lemma}

\begin{proof}
We first note that $\gamma(2) = \big( S_n(2) - S_{n-1}(2) \big)^5 = 1$ and hence $z-2 \nmid \gamma$.
Assume $\gcd(\alpha, \beta, \gamma)$ has a nontrivial prime factor $P\in \C[z]$. 

Since $P \mid\alpha$, there are 2 cases to consider: $P \mid X$ and $P \mid X-Y$.
If $P \mid X-Y$ then $P \mid (z-2)^2 X^5$, since 
$\gamma \equiv (z-2)^2 X^5 \pmod{X-Y}$. This implies that $P \mid X$. 
Hence $P \mid \gcd(X,Y) =1$ by Lemma \ref{gcd}, which contradicts the assumption. 
If $P \mid X$ then $P \mid (z-1) Y^5$, since $\gamma \equiv - (z-1) Y^5 \pmod{X}$.
Since $\gcd(X,Y) =1$, we must have $P \mid z-1$. From Lemma \ref{chev} we see that $z-1 \mid S_n(z)$ if and only if 
$n \equiv 2 \pmod{3}$. The lemma follows, since $X=S_n(z)$ does not have any repeated factors 
by Lemma \ref{chev}.
\end{proof}

\begin{lemma} \label{lem2}
$R(x,z)$ does not have any prime factors of degree $1,2,3$ in $x$.
\end{lemma}

\begin{proof}
We first note that it suffices to show that $R =  R(x,z)$ does not have any prime factors of degree $1,2$ in $x$.
 
Assume $R$ has a prime factor $P = a x^2 + bx +c$ of degree $2$ in $x$, where $a,b,c \in \C[z]$ and $a \not= 0$. 
There are 2 cases to consider: $b \not= 0$ and $b=0$.

Suppose $b \not= 0$. Since $R$ is even in $x$, $Q = ax^2 - bx +c $ is also a prime factor of $R$. 
This implies that $PQ = a^2 x^4 + (2ac-b^2) x^2 + c^2$ is a factor of $R$. 
In particular, we have 
\[ 
\frac{a^2}{c^2} = \frac{\alpha}{\gamma} = (z-2) \frac{\alpha''}{\gamma''}
\] 
where neither $\alpha''$ nor $\gamma''$ is divisible by $z-2$. 
This cannot occur, since the left hand side is a square in $\C(z)$.

Suppose $b=0$. In this case we must have $R = (ax^2 + c) (dx^2 +f)$ for some $d, f \in \C[z]$. 
This implies that $\alpha = ad$, $\beta = af + cd$ and $\gamma = cf$. 
In particular, we have $\beta^2 - 4 \alpha \gamma = (af-cd)^2$. 
By a direct computation we have
$$\beta^2 - 4 \alpha \gamma = (z-2) X^2  (X-Y)^2 \left(X^2 z-6 X^2+8 X Y-4 Y^2\right) \left(X^2-X Y z+Y^2\right)^2.$$
By the above argument, this should be a square in $\C[z]$ and 
so $z-2$ must divide $X^2 z-6 X^2+8 X Y-4 Y^2$. Then we obtain a contradiction, since  
$$X^2 z-6 X^2+8 X Y-4 Y^2 = (z-2) X^2 - 4(X-Y)^2$$
is not divisible by $z-2$. Note that $(X-Y) \mid_{z=2} \, = S_n(2) - S_{n-1}(2) =1 \not= 0$.

Assume $R$ has a prime factor $P' = a x + b$ of degree $1$ in $x$, where $a,b \in \C[z]$ and $a \not= 0$. 
Note that $b \not= 0$, as $\gamma \not= 0$.  
Since $R$ is even in $x$, $Q' = ax -b$ is also a prime factor of $R$. 
This implies that $P'Q' = a^2 x^2 - b^2$ is a factor of $R$. 
Although $a^2 x^2 - b^2$ is not a prime factor of $R$, but 
we can show that it is not a factor of $R$ by a similar argument to the previous case. This completes the proof of the lemma.
\end{proof}

Lemmas \ref{lem1} and \ref{lem2} imply the following.

\begin{proposition} \label{factor}
If $n \not\equiv 2 \pmod{3}$ then $R(x,z)$ is irreducible in $\C[x,z]$.
If $n \equiv 2 \pmod{3}$ then $R(x,z) = (z-1) R'(x,z)$ where $R'(x,z)$ is irreducible in $\C[x,z]$.
\end{proposition}

Proposition \ref{factor} implies Theorem \ref{thm}.

On the other hand, 
we see that two-bridge knots $\fb(30 k + 3,5)$ and $\fb(30 k - 3,5)$ 
are not minimal as follows. 
The rational numbers $5/(30k+3)$ and $5/(30k-3)$ admit 
the following continued fraction expansions: 
\begin{align*}
\frac{5}{30k+3} 
& = [6k,2,-3] 
 = [\underbrace{3,0,3,0,\ldots,0,3}_{4k-1},2,-3], \\
\frac{5}{30k-3} 
& = [6k,-2,3] 
 = [\underbrace{3,0,3,0,\ldots,0,3}_{4k-1},-2,3] .  
\end{align*}
This implies that there exist epimorphisms 
from their knot groups onto the trefoil knot group 
by Ohtsuki-Riley-Sakuma construction \cite{ORS} (see also \cite{szk}). 

Finally, two-bridge knots $\fb(6k+1,6)$ and $\fb(6k+5,6)$ are minimal 
by Corollary \ref{cor:kpplusminusone}. 

As a consequence of the above arguments, we obtain the following theorem. 

\begin{theorem}\label{thm:mainthmA}
For $q \leq 6$, 
a two-bridge knot $\fb(p,q)$ is not minimal 
if and only if one of the following two conditions is satisfied: \\
$(1)$ $p$ is not prime and $q = 1$, \\
$(2)$ $p \equiv \pm 3 \pmod{30}$ and $q = 5$.
\end{theorem}

\section{Main theorem II}

In this section, we determine which two-bridge knot $\fb(p,q)$ is not minimal where $p \leq 100$. 

As mentioned in Section \ref{sect:conditions}, 
if either $p$ is not divisible by $p'$ or $\Delta_{\fb(p,q)}(t)$ is not divisible by $\Delta_{\fb(p',q')}(t)$, 
then there does not exist an epimorphism from $G(\fb(p,q))$ onto $G(\fb(p',q'))$. 

In \cite{szk}, a relationship between epimorphisms of two-bridge knot groups and their crossing numbers is studied. 
To be precise, 
if there exists an epimorphism from the knot group of a two-bridge knot $K$ 
onto that of another knot $K'$, 
then the crossing number of $K$ is greater than or equal to three times of that of $K'$. 
%By using these results, we can determine the existence of an epimorphism between two-bridge knot groups. 
For example, $\fb(45,14)$ is minimal as follows. 
Pairs $(p,q)$ satisfying that $p$ is a divisor of $45$ and 
$\Delta_{\fb(45,14)}(t)$ is divisible by $\Delta_{\fb(p,q)}(t)$ are only $(5,1)$ and $(9,2)$. 
However,  
%By theorem \ref{thm:bb}, 
%although $\Delta_{\fb(45,14)}(t)$ is divisible by $\Delta_{\fb(5,1)}(t)$ and $\Delta_{\fb(9,2)}(t)$, 
the crossing number of $\fb(45,14)$ is $10$ 
which is less than three times of the crossing numbers of $\fb(5,1)$ and $\fb(9,2)$. 
Hence there exists an epimorphism from $G(\fb(45,14))$ onto neither $G(\fb(5,1))$ nor $G(\fb(9,2))$. 

Furthermore,  in \cite{szk} all epimorphisms between two-bridge knot groups 
with up to $30$ crossings are determined.
For example, by using this result, 
there does not exist an epimorphism from $G(\fb(51,16))$ onto $G(\fb(3,1))$, 
although $\Delta_{\fb(51,16)}(t)$ is divisible by $\Delta_{\fb(3,1)}(t)$ and 
the crossing number of $\fb(51,16)$ is $11$ 
which is bigger than three times of the crossing number of $\fb(3,1)$. 

By these arguments, we obtain the following theorem. 

\begin{theorem}\label{thm:mainthmB}
For $p \leq 100$, 
a two-bridge knot $\fb(p,q)$ is not minimal 
if and only if $(p,q)$ is one of the following pairs: 
\begin{align*}
& (9,1), (15,1), (21,1), (27,1), (27,5), (33,1), (33,5), (35,1), (39,1), (39,7), (45,1), \\
& (45,7), (45,19), (49,1), (51,1), (55,1), (57,1), (57,5), (63,1), (63,5), (63,11), (65,1), \\
& (69,1), (69,11), (69,19), (75,1), (75,13), (75,29), (77,1), (81,1), (81,7), (81,13), \\
& (85,1), (85,9), (85,38), (87,1), (87,5), (87,7), (91,1), (93,1), (93,5), (93,11), (95,1), \\
& (95,9), (99,1), (99,17), (99,29).
\end{align*}
\end{theorem}

\section*{Acknowledgements}
The first and second authors were partially supported by KAKENHI 
(No.~26800046 and No.~16K05159), Japan Society for the Promotion of Science, Japan. 
The third author was partially supported by a grant from the Simons Foundation (No.~354595
to AT).

\end{document}